\documentclass[12pt,a4paper]{article}

\usepackage{amssymb, amsmath, amsthm}

\usepackage{algorithm}
\usepackage{algpseudocode}

\usepackage[cm]{fullpage}
\usepackage[english]{babel}
\usepackage[pdftex]{graphicx}
\usepackage{booktabs}
\usepackage{xcolor}

\usepackage{hyperref}

\newtheorem{thm}{Theorem}

\newtheorem{prop}{Proposition}

\title{Time-fractional porous medium equation: Erd\'elyi-Kober integral equations, compactly supported solutions, and numerical methods}
\author{
Belen L\'opez\thanks{Departamento de Matem\'aticas, Universidad de Las Palmas de Gran Canaria,	Campus de Tafira Baja, $35017$ Las Palmas de Gran Canaria, Spain.},
Hanna Okrasi{\'n}ska-P{\l}ociniczak\thanks{Department of Mathematics, Wroclaw University of Environmental and Life Sciences, ul. C.K. Norwida 25, 50-275 Wroclaw, Poland}, \\
\L ukasz P\l ociniczak\thanks{Faculty of Pure and Applied Mathematics, Wroc{\l}aw University of Science and Technology, Poland, \underline{corresponding author:} lukasz.plociniczak@pwr.edu.pl},
Juan Rocha$^*$
}
\date{}

\begin{document}
\maketitle

\begin{abstract}
	The time-fractional porous medium equation is an important model of many hydrological, physical, and chemical flows. We study its self-similar solutions, which make up the profiles of many important experimentally measured situations. We prove that there is a unique solution to the general initial-boundary value problem in the one-dimensional setting. When supplemented with boundary conditions from the physical models, the problem exhibits a self-similar solution described with the use of the Erd\'elyi-Kober fractional operator. Using a backward shooting method, we show that there exists a unique solution to our problem. 
	
	The shooting method is not only useful in deriving the theoretical results. We utilize it to devise an efficient numerical scheme to solve the governing problem along with two ways of discretizing the Erd\'elyi-Kober fractional derivative. Since the latter is a nonlocal operator, its numerical realization has to include some truncation. We find the correct truncation regime and prove several error estimates. Furthermore, the backward shooting method can be used to solve the main problem, and we provide a convergence proof. The main difficulty lies in the degeneracy of the diffusivity. We overcome it with some regularization. Our findings are supplemented with numerical simulations that verify the theoretical findings. \\
	
	
\end{abstract}

\section{Introduction}
Our main interest is the following time-fractional porous medium problem on the half-line
\begin{equation}
	\label{eqn:MainPDE}
	\begin{cases}
		\partial_t^\alpha u = \left(D(u)u_x\right)_x, & x \in\mathbb{R}_+, \; t\in (0,T), \; \alpha\in(0,1)\\
		u(x,0) = 0, & x \in\mathbb{R}_+, \\
		u(0,t) = M, & t \in (0,T),
	\end{cases}
\end{equation}
where $T>0$ is the final time, $M>0$ the initial value, while the Caputo fractional derivative is defined with the help of the fractional integral $I^\alpha_t$
\begin{equation}
	\label{eqn:Caputo}
	\partial_t^\alpha u(x,t) := I^{1-\alpha}_t u_t(x,t), \quad I^\alpha_t u(x,t) = \frac{1}{\Gamma(\alpha)}\int_0^t (t-s)^{\alpha-1} u(x,s) ds, \quad \alpha\in(0,1).
\end{equation}
For the diffusivity $D$ we assume that it is a $C^1(\mathbb{R}_+)$ function with
\begin{equation}\label{eqn:Diffusivity}
	D(0) = 0, \quad D(u) > 0, \quad D'(u) > 0, \quad u > 0. 
\end{equation}
The most important example is the typical power type (or Brookes-Correy model, as known in hydrology \cite{brooks1965hydraulic}) diffusivity for which $D(u) \propto u^m$ with $m>0$. Note that we assume the \emph{degeneracy}, that we allow for a situation when $D(0) = 0$. This has a profound effect on the solution and is the main reason for the finite speed of propagation (for a comprehensive treatment see \cite{vazquez1992introduction,vazquez2007porous}). The mathematical treatment of the considered PDE (\ref{eqn:MainPDE}) has attracted some recent attention, especially for the time-fractional case. For example, bounded weak solutions of the degenerate and non-degenerate cases have been found in \cite{wittbold2021bounded} in the case of vanishing Dirichlet data and the problem on a bounded domain $x\in\Omega\subset\mathbb{R}$. A very general abstract study of the associated problem has also been given in \cite{akagi2019fractional}. Moreover, in \cite{Dji18} a similar problem has been considered in the full space $x\in\mathbb{R}^d$ and the authors have proved the existence and uniqueness of a complactly supported solution provided that the initial data have this property. It seems that the problem on the half-line has not been investigated adequately in the time-fractional setting. In this paper, we present some further analytical and numerical results that go beyond our initial works \cite{plociniczak2018existence,plociniczak2019numerical,plociniczak2014approximation}. Our main observation is based on the fact that the considered initial and boundary data are self-similar, which allow for a transformation of the governing PDE into an ordinary integro-differential equation. The main evolution operator then becomes the Erd\'elyi-Kober (EK) fractional operator that has previously been found in complex analysis and special functions \cite{kiryakova1997transmutation,sneddon1975use}. Furthermore, when analyzing stochastic processes derived from fractional Brownian motion, the EK operator becomes the main driving force to evolve diffusive dynamics \cite{pagnini2012erdelyi}. 

The problem (\ref{eqn:MainPDE}) models moisture imbibition in the essentially one-dimensional porous medium with the boundary. The initially dry domain is suddenly exposed to a constant concentration of moisture at the boundary. This is a typical setting for measuring the properties of the medium \cite{wang2011mesoscale}. In this setup, the obvious type of solution to look for is the self-similar profile, which is a function of the self-similar variable $x/t^{1/2}$. However, as some new experiments on modern materials show, diffusion can undergo slower (sub-diffusive) and faster (super-diffusive) dynamics \cite{El04,Kun01,lockington2003anomalous,ren2021quantifying,wu2019anomalous,Sun13,El20,zhao2022hydraulic}. Further examples come from biology \cite{Lev97,Sun17}, physics \cite{Sch97,Del05}, and other fields of science. In \cite{Ger10} it was suggested that the time-fractional Caputo derivative is adequate for modeling such a phenomenon. This approach was successful in the sense that the modeling outcome reproduced the experimental data with good precision. A physical derivation of the model in the hydrological setting is given in \cite{plociniczak2015analytical}.

There is a substantial amount of various numerical methods for the diffusion with the time-fractional derivative (for a review, see \cite{Die02,garrappa2018numerical,li2015numerical}). We would like to stress the fact that most of the results consider linear and only space-dependent diffusivity $D=D(x)$. The approaches are based on different kinds of Caputo discretization method and spatial scheme. For example, the reader can consult several approaches in \cite{Kop19,Lan05} for the so-called L1 Caputo discretization scheme and convolution quadrature \cite{schadle2006fast,jin2017correction}. Finite difference methods were considered in \cite{stynes2017error}, finite elements in \cite{liao2019discrete,ford2011finite}, and spectral methods in \cite{lin2007finite,li2009space}. Finding the error estimates for the case with nonsmooth initial data requires some additional care and it is crucial for many applications. This is especially relevant for time-dependent diffusivity \cite{jin2019subdiffusion,Mus18,Plo22a}. As for some excursions from linearity, there are some recent methods concerning the semilinear diffusion in \cite{jin2018numerical,al2019numerical,karaa2020galerkin}. The quasilinear case is just beginning to be investigated, and we can refer the reader to our previous work concerning this important problem \cite{plociniczak2022linear,plociniczak2022linearClimate}. We would like to stress the fact that most of these results were considered only the non-degenerate case for which $D \geq D_0 > 0$. We have developed several numerical approaches for the time-fractional degenerate case in some of our previous work. For example, in \cite{plociniczak2019numerical} a fast quadrature was devised to solve \ref{eqn:MainPDE}. This result was further refined to yield a second-order scheme in \cite{okrasinska2022second,okrasinski1993power}. Since the nonlocal nature of the Caputo (or EK) operator increases the computational cost of all of the numerical methods used to approximate them, we believe that there is a need for developing fast and accurate schemes for solving equations governed by these operators, especially in the degenerate case. 

This paper has the following structure. In the next section we give a short proof of the uniqueness of a general time-fractional porous medium problem in $\Omega\subseteq\mathbb{R}$, where $\Omega$ can be bounded or not. The idea of the proof is to modify the classical approach to the Caputo derivative setting. Having the uniqueness, we proceed in Section 3 to the study of self-similar solutions of the main problem (\ref{eqn:MainPDE}). We use the backward shooting method to prove the existence of such solutions and combine it with the previous uniqueness result. In Section 4 we use the hints of our analytical reasoning to devise efficient numerical methods for approximating the Erd\'elyi-Kober fractional operator and the exact solution of (\ref{eqn:MainPDE}). There we also present several error estimates and the convergence proof. We illustrate the theory by several numerical examples.  

\section{Uniqueness of the weak solution}
Although we are mostly interested in (\ref{eqn:MainPDE}) as the model of moisture imbibition in porous media, in this section we consider a general problem for which we present the proof of uniqueness. To this end, consider
\begin{equation}\label{eqn:MainPDEGeneral}
	\begin{cases}
		\partial_t^\alpha u = \left(D(u)u_x\right)_x, & x \in(a,b), \; t\in (0,T), \; \alpha\in(0,1)\\
		u(x,0) = u_0(x), & x \in(a,b), \\
		u(a,t) = \mu(t), \; u(b,t) = \nu(t),
		& t \in (0,T),
	\end{cases}	
\end{equation}
in which we allow for a general initial and boundary conditions. Define $\Omega_T := (a,b) \times (0,T)$ as the domain of the solution. Moreover, let
\begin{equation}\label{eqn:DiffusivityIntegrated}
	K(z) := \int_0^z D(s) ds,
\end{equation}
then the PDE can be written as $\partial^\alpha_t u = (K(u))_{xx}$. As in the classical case, we cannot expect that the above will enjoy classical solutions but rather weak ones (although there are other options \cite{vazquez2007porous,wittbold2021bounded}). Let $\chi \in C^1 (\Omega_T)$ be the arbitrary test function. By multiplication and integration of the PDE we obtain the following
\begin{equation}
	\int_a^b \int_0^T \partial^\alpha_t u \; \chi dt dx - \int_a^b \int_0^T (K(u))_{xx} \chi dt dx = 0.
\end{equation}
Now, we integrate by parts to move the derivatives into test functions. First, for the time derivative by (\ref{eqn:Caputo}) and Fubini's theorem we have
\begin{equation}\label{eqn:RFracInt0}
	\begin{split}
		\int_0^T \partial^\alpha_t u \; \chi dt 
		&= \int_0^T (I^{1-\alpha}_t u_t) \chi dt =  \frac{1}{\Gamma(1-\alpha)} \int_0^T \left(\int_{0}^t (t-s)^{-\alpha} u_s(x,s) ds\right) \chi(x,t) dt \\
		&=  \frac{1}{\Gamma(1-\alpha)} \int_0^T \left( \int_s^T (t-s)^{-\alpha} \chi(x,t) dt \right) u_s(x,s) ds = \int_0^T u_s (J^{1-\alpha}_s \chi) ds,
	\end{split}
\end{equation}
where we define the right-sided fractional integral,
\begin{equation}\label{eqn:RFracInt}
	J^\alpha_t u(x,t) = \frac{1}{\Gamma(\alpha)} \int_t^T (s-t)^{\alpha-1} u(x,s) ds, \quad \alpha > 0.
\end{equation}
Now, by integrating by parts and renaming the integration variable, we can obtain
\begin{equation}
	\int_0^T \partial^\alpha_t u \; \chi dt = u_0(x) J^{1-\alpha}_t \chi(x,0) - \int_0^T u \frac{\partial}{\partial t}(J^{1-\alpha}_t \chi) dt.
\end{equation}
For the space derivatives the integration by parts along with boundary conditions gives the following
\begin{equation}
	\int_a^b (K(u))_{xx} \chi dx = D(\nu(t))\chi(b,t) - D(\mu(t))\chi(a,t) - \int_a^b (K(u))_{x} \chi_x dx.
\end{equation}
Therefore, we define the \emph{weak solution} of (\ref{eqn:MainPDEGeneral}) as a $H^1(\Omega_T)$ function that satisfies
\begin{equation}\label{eqn:WeakSolution}
	\begin{split}
		\int_0^\infty \int_0^T  &\left[(K(u))_{x} \chi_x - u \frac{\partial}{\partial t}(J^{1-\alpha}_t \chi) \right] dtdx \\
		&=\int_0^T \left[D(\mu(t))\chi(a,t) - D(\nu(t))\chi(b,t)\right] dt - \int_a^b u_0(x) J^{1-\alpha}_t\chi(x,0) dx, \quad \chi \in C^1(\Omega_T). 
	\end{split}
\end{equation}
This approach to defining the weak solution mimics the classical case with $\alpha = 1$. Some relevant existence results are given in \cite{vazquez2007porous}.

We now turn to the uniqueness proof. It is a generalization of the simple and elegant argument originally devised by Ladyzenskaya \cite{ladyzenskaya1967linear} and later frequently used in many cases \cite{vazquez2007porous}. 

\begin{thm}\label{thm:Uniqueness}
	There can be at most one solution to (\ref{eqn:WeakSolution}). 
\end{thm}
\begin{proof}
	Assume that $u_1$ and $u_2$ are some solutions to (\ref{eqn:WeakSolution}). Taking the difference cancels the initial and boundary conditions, yielding
	\begin{equation}\label{eqn:Uniquness0}
		\int_a^b \int_0^T \left[\left((K(u_1))_{x} - (K(u_2))_x\right) \chi_x - (u_1-u_2) \frac{\partial}{\partial t}(J^{1-\alpha}_t \chi) \right] dtdx = 0
	\end{equation}
	for each $\chi \in C^1(\Omega_T)$. Now, choose a test function of the form
	\begin{equation}
		\chi = J^\alpha_t (K(u_1)-K(u_2)).
	\end{equation}
	Of course, the above choice may have to have sufficient regularity in order to make it a test function; however, by a standard mollification argument, we can ascertain that it is admissible. The details of such a procedure are described in detail in \cite{vazquez2007porous} and, hence, we omit them here. 
	
	The integral with the gradient is non-negative. This can be seen by using Fubini's theorem just as in (\ref{eqn:RFracInt0}) to move the right-sided fractional integral $J^\alpha$ into the usual (left-sided) integral $I^\alpha$ and invoking Lemma 3.1 (ii) from \cite{mustapha2014well}
	\begin{equation}
		\begin{split}
			\int_a^b \int_0^T (K(u_1)-K(u_2))_x J^\alpha_t (K(u_1)-K(u_2))_x dt dx = \\ \int_a^b \int_0^T (K(u_1)-K(u_2))_x I^\alpha_t (K(u_1)-K(u_2))_x dt dx 
			&= \int_a^b \int_0^T |I^{\alpha/2}(K(u_1)-K(u_2))_x|^2 dt dx \geq 0.
		\end{split}
	\end{equation}
	Now, in the other integral we use the semigroup property of the fractional integral, that is, $J^\mu J^\nu = J^{\nu+\mu} = J^{\nu}J^{\nu}$ to obtain the following 
	\begin{equation}
		\begin{split}
			-&\int_a^b \int_0^T (u_1-u_2) \frac{\partial}{\partial t}(J^{1-\alpha}_t J^\alpha_t (K(u_1)-K(u_2))) dt dx \\
			&= -\int_a^b \int_0^T (u_1-u_2) \frac{\partial}{\partial t}(J^{1} (K(u_1)-K(u_2))) dt dx = \int_a^b \int_0^T (u_1-u_2)(K(u_1)-K(u_2)) dt dx\geq 0
		\end{split}
	\end{equation}
	since $K$ is increasing, and therefore $K(u_1)-K(u_2)$ has the same sign as $u_1-u_2$. Therefore, (\ref{eqn:Uniquness0}) transforms into
	\begin{equation}
		\int_a^b \int_0^T \left[ |I^{\alpha/2}(K(u_1)-K(u_2))_x|^2 + (u_1-u_2) (K(u_1)-K(u_2)) \right] dtdx = 0.
	\end{equation}
	Because the integrand is non-negative we must have $u_1 = u_2$. This ends the proof. 
\end{proof}

\section{Self-similar solution}
Now we can go back to the original problem (\ref{eqn:MainPDE}) and look for the solution in a self-similar form \cite{plociniczak2018existence,plociniczak2019numerical}
\begin{equation}
	u(x,t) = U(\eta), \quad \eta := x t^{-\frac{\alpha}{2}}. 
\end{equation}
Standard calculations (see, for example, \cite{plociniczak2019numerical}) lead to the ordinary equation for the unknown profile $U=U(\eta)$
\begin{equation}
	\label{eqn:PorousMediumD}
	\left(D(U)U'\right)' = \left[A - B\eta \frac{d}{d\eta}\right] F_\alpha U, \quad 0 < \eta < \infty, 
\end{equation}
where $A=1-\alpha$ and $B=\alpha/2$. Note that for different types of boundary conditions, we obtain different values of the constants $A$ and $B$, however, the structure of the above equation stays the same. This is why we decided to leave general constants appearing in the governing equation. The operator $F_\alpha$ is a particular version of the Erdl\'elyi-Kober fractional operator \cite{sneddon1975use}
\begin{equation}
	\label{eqn:EK}
	F_\alpha U(\eta) = \frac{1}{\Gamma(1-\alpha)} \int_0^1 (1-s)^{-\alpha} U(s^{-B}\eta) ds, \quad 0<\alpha< 1.
\end{equation}	
The boundary conditions are
\begin{equation}
	U(0) = M, \quad U(\infty) = 0.
\end{equation}
By Theorem \ref{thm:Uniqueness} we know that the original problem (\ref{eqn:MainPDE}) has a unique solution and it is of the self-similar form. In \cite{plociniczak2018existence} it has been proved that for the diffusivity of the porous medium, that is $D(U) \propto U^m$ for $m\geq 1$, there exists a \emph{compactly supported} solution. We now know that it is precisely the unique solution of (\ref{eqn:MainPDE}). What remains is to consider the general case $D=D(U)$ satisfying (\ref{eqn:Diffusivity}). In what follows, we present some initial results on this topic. 

There is no straightforward way to solve (\ref{eqn:PorousMediumD}). In \cite{plociniczak2018existence} we have adopted a certain transformation in the case of power-type diffusivity to obtain a Volterra integral equation for which the theory is known. Since in the general case no such transformation is available, we follow a different route. The idea is to use \emph{backward shooting method} - the idea that was used in the classical setting \cite{atkinson1971similarity}. However, in this non-local version, it gains much more depth and meaning. To wit, we assume that we are looking for a completely supported solution with support $[0,\eta^*]$ for fixed $\eta^*> 0$. That is, we have $U(\eta) = 0$ for $\eta \geq \eta^*$. We then consider \emph{initial value problem} starting from $\eta=\eta^*$ going backward to $\eta = 0$. Finally, by varying $\eta^*$ we can adjust the value of the solution at $\eta=0$ in order to have $U(0) = M$. This is precisely the idea of the numerical method presented in the next section. The crucial observation is to note that in the definition of the EK operator (\ref{eqn:EK}) the function that is operated is evaluated at $s^{-\alpha/2}\eta$ for fixed $\eta > 0$. But since, for $s\in (0,1)$ we always have $s^{-\alpha/2} > 1$ it appears that imposing an initial value would not yield an amenable problem. That is to say, in order to compute $F_\alpha U(\eta)$ for $\eta$ close to $0$ we would require knowledge of $U$ over almost the whole half-line, which is certainly not admissible in a step-by-step calculation. In other words, we cannot advance our solution from $\eta = 0$ onward. 

First, let us observe how the solution to our problem can behave. The following result states many a priori properties of the solution.

\begin{prop}
	Fix $\eta^* > 0$. Let $U=U(\eta)$ be a solution of (\ref{eqn:PorousMediumD}) with (\ref{eqn:Diffusivity}) in the left neighborhood of $\eta^*$ such that $U(\eta) = 0$ for $\eta \geq \eta^*$ . Then, it is positive, decreasing, and the following holds
	\begin{equation}
		\label{eqn:DiffusivityInt}
		\int_0^1 \frac{D(s)}{s} ds < \infty,
	\end{equation}
	along with 
	\begin{equation}
		\label{eqn:NoFlux}
		\lim\limits_{\eta\rightarrow\eta^*} -D(U)U'(\eta) = 0.	
	\end{equation}
\end{prop} 
\begin{proof}
	Since $U(\eta) = 0$ for $\eta\geq \eta^*$ we can integrate (\ref{eqn:PorousMediumD}) from $\eta$ to $\eta^*$ to obtain
	\begin{equation}
		-D(U)U'(\eta) = A \int_\eta^{\eta^*} F_\alpha U(z) dz - B \int_\eta^{\eta^*} z \left(F_\alpha U\right)'(z)dz = (A+B) \int_\eta^{\eta^*} F_\alpha U(z) dz + B \eta F_\alpha U(\eta),
	\end{equation}
	where in the second equality, we have calculated by parts and used the definition of the EK operator (\ref{eqn:EK}). From this it immediately follows that $U'(\eta) < 0$ and $U(\eta) \geq 0$ which by (\ref{eqn:EK}) implies that
	\begin{equation}\label{eqn:EKBound}
		F_\alpha U(\eta) \leq \frac{U(\eta)}{\Gamma(2-\alpha)}. 
	\end{equation}
	Moreover, letting $\eta \rightarrow \eta^*$ we obtain the no-flux condition (\ref{eqn:NoFlux}). Hence, using the monotonicity and going back to the integrated equation, we have the following
	\begin{equation}
		\begin{split}
			-D(U)U'(\eta) &\leq \frac{A}{\Gamma(2-\alpha)} \eta^* U(\eta) - B \eta^* \int_\eta^{\eta^*} \left(F_\alpha U\right)'(z)dz \\
			&= \frac{A}{\Gamma(2-\alpha)} \eta^* U(\eta) + B \eta^* F_\alpha U(\eta) \leq \frac{A+B}{\Gamma(2-\alpha)}\eta^* U(\eta).
		\end{split}
	\end{equation}
	If we now divide by $U(\eta)$ and integrate from arbitrary $\eta_1$ to $\eta_2$ we obtain
	\begin{equation}
		-\int_{\eta_1}^{\eta_2} \frac{D(U(z))U'(z)}{U(z)} dz \leq \frac{A+B}{\Gamma(2-\alpha)}\eta^* \left(\eta_2-\eta_1\right) \leq \frac{A+B}{\Gamma(2-\alpha)}(\eta^*)^2.
	\end{equation}
	A change of the variable $s = U(z)$ along with the monotonicity of $U$ lets us write
	\begin{equation}\label{eqn:DEst}
		\int_{U(\eta_2)}^{U(\eta_1)} \frac{D(s)}{s} ds < \frac{A+B}{\Gamma(2-\alpha)}(\eta^*)^2.
	\end{equation}
	Letting $\eta_2 \rightarrow \eta^*$ implies $U(\eta_2) \rightarrow 0$ which concludes the proof. 
\end{proof}
From the above proof, we see that our compactly supported solution confirms every physical intuition: it is a positive, bounded function with finite speed of propagation and vanishing flux at the interface. The condition (\ref{eqn:DiffusivityInt}) is necessary for the existence of the compact support. For example, with power-type diffusion, we have
\begin{equation}
	\int_0^1 \frac{s^m}{s} ds < \infty \quad \text{iff} \quad m > 0,
\end{equation}
which confirms our previous results. 

To proceed further we integrate (\ref{eqn:PorousMediumD}) twice from $\eta$ to $\eta^*$, use the vanishing boundary conditions at $\eta = \eta^*$, and integrate by parts to arrive at the integral equation
\begin{equation}
	\label{eqn:PorousMediumInt}
	K(U(\eta)) = \int_{\eta}^{\eta^*} \left((A+B)(z-\eta) + B z\right)F_\alpha U(z) dz =: \int_{\eta}^{\eta^*} G(\eta, z)F_\alpha U(z) dz,
\end{equation} 
In the following we present the main existence result.

\begin{thm}\label{thm:Existence}
	Fix $\eta^* > 0$. There exists a solution to (\ref{eqn:PorousMediumD}) with $U(\eta) = 0$ for $\eta\geq \eta^*$ and (\ref{eqn:NoFlux}). 
\end{thm}
\begin{proof}
	We will apply the Leray-Schauder fixed point theorem. First, since $D$ is increasing, the function $K$ defined in (\ref{eqn:DiffusivityIntegrated}) is convex. Therefore, there exists a unique positive solution $x_\lambda$ to the equation $K(x) = \lambda x$ for any $\lambda > 0$. Let $X = C[0,\eta^*]$ be the Banach space of continuous functions on $[0,\eta^*]$ with the norm $\|U\| := \max_{0\leq \eta\leq \eta^*} |U(\eta)|$. From (\ref{eqn:PorousMediumInt}) we can obtain the a priori bound for the solution. First, by the fundamental estimate of the Erd\'elyi-Kober operator (\ref{eqn:EKBound}) we have the following
	\begin{equation}
		K(U(\eta)) \leq (A+2B)\eta^*\int_0^{\eta^*} F_\alpha U(z) dz \leq \frac{A+2B}{\Gamma(2-\alpha)} (\eta^*)^2 \|U\|. 
	\end{equation}
	By taking the maximum on the left-hand side and using the continuity of $U$ we further have the following
	\begin{equation}
		K(\|U\|) \leq \frac{A+2B}{\Gamma(2-\alpha)} (\eta^*)^2 \|U\|.
	\end{equation}
	Since $K$ is convex, from simple geometrical considerations, we must have $\|U\| \leq x_\lambda$ with $\lambda := (A+2B)/\Gamma(2-\alpha) (\eta^*)^2$. Therefore, we have the a priori upper bound for any solution to (\ref{eqn:PorousMediumD}). 
	
	Having the bound for the solution, we define the operator $N: X \mapsto X$ by the formula
	\begin{equation}
		N(y)(\eta) = \int_{\eta}^{\eta^*} G(\eta, z)F_\alpha K^{-1} y(z) dz,
	\end{equation}
	which is well defined because $K$ is monotone and hence $K^{-1}$ exists. If $y$ is the fixed point of $N$, then $U = K^{-1}y$ will be the solution of (\ref{eqn:PorousMediumD}). Since the integrand in the definition of $N$ is a continuous function of $\eta$, $z$, and $y$ defined in a bounded and closed set ($y$ is bounded), the operator $N$ is compact. Therefore, by the standard version of the Leray-Schauder theorem (for ex. Theorem 6.A in \cite{zeidlernonlinear}), that is, a priori bounded and compact operator has a fixed point, we conclude that problem (\ref{eqn:PorousMediumD}) has a solution. 
\end{proof}

Now we know that for each $\eta^*$ there exists a complactly supported solution of our problem that is a bounded decreasing function. In the following we show that, at least for small $U(0)=M$, we can determine that there is a $\eta^*$ such that the solution attains $M$ for $\eta=0$. 

\begin{prop}
	Assume (\ref{eqn:Diffusivity}) and (\ref{eqn:DiffusivityInt}). For sufficiently small $M>0$ there exists a unique solution to equation (\ref{eqn:PorousMediumD}) with $U(0) = M$ and $U(\eta) = 0$ for $\eta\geq \eta^*$. 
\end{prop}
\begin{proof}
	Define the continuous function $f(\eta^*) := U(0;\eta^*)$. Our goal is to show that the function $\eta^* \mapsto f(\eta^*) - M$ has exactly one zero. If we take $\eta_1 \rightarrow 0$ and $\eta_2\rightarrow\eta^*$ in (\ref{eqn:DEst}) we obtain
	\begin{equation}
		\int_0^{f(\eta^*)} \frac{D(s)}{s} ds \leq \frac{A+B}{\Gamma(2-\alpha)}(\eta^*)^2.
	\end{equation}
	Therefore, by assumption we have $f(\eta^*) \rightarrow 0$ for $\eta^* \rightarrow 0$. If $M$ is small enough, then there exists $\eta^*$ such that $f(\eta*) = M$ and the existence is proved.
	
	In order to ascertain uniqueness, we will show that $f$ is an increasing function. To this end, assume on the contrary that there are $\eta_1 < \eta_2$ such that $f(\eta_1) \geq f(\eta_2)$. From the monotonicity of the solution $U$ it follows that there exist $\eta_0$ such that $U_1(\eta_0) = U_2(\eta_0)$ and $U_1(\eta) < U_2(\eta)$ for $\eta_0 < \eta \leq \eta_1$. Here, we have denoted $U_i$ as the solution to the problem with the support $[0,\eta_i]$. We find that each $U_i$ satisfies (\ref{eqn:PorousMediumInt}). Subtracting the respective equations from $\eta=\eta_0$ we obtain the following
	\begin{equation}
		0 = \int_{\eta_0}^{\eta_1} G(\eta_0,z) \left(F_\alpha U_2(z) - F_\alpha U_1(z)\right)dz + \int_{\eta_1}^{\eta_2} G(\eta_0, z) F_\alpha U_2(z) dz. 
	\end{equation}
	However, both above terms are strictly positive and, hence, we arrive at a contradiction. The function $f$ is strictly increasing. 
\end{proof}

Therefore, assuming (\ref{eqn:Diffusivity}) and (\ref{eqn:DEst}) we have proved that there exists a self-similar solution to the problem (\ref{eqn:MainPDE}) and by Theorem \ref{thm:Uniqueness} we know that it is unique.

\section{Numerical methods}

In this section, we develop an efficient numerical method for both approximating the Erd\'elyi-Kober fractional operator (\ref{eqn:EK}) and solving the main time-fractional porous medium equation (\ref{eqn:MainPDE}). 

\subsection{Erdelyi-Kober operator}
We start by numerically approximating the EK operator (\ref{eqn:EK}) acting on \textit{any bounded function} $U: [0,\infty) \mapsto \mathbb{R}_+$. Note that for a moment we do not assume that $U$ satisfies (\ref{eqn:PorousMediumInt}) nor has a compact support. Some general quadratures for the EK operator have been analyzed in \cite{plociniczak2017numerical} where a thorough error analysis has also been given. Here, we focus only on the main difference between the paper cited: here, our operator involves the solution evaluated at $s^{-B}\eta$. In \cite{plociniczak2017numerical} only the positive exponent case was considered. This furnishes a radical change in the numerical analysis. To provide a concrete examples, we devise two schemes: of first (rectangle) and second (trapezoid) order. First, introduce a uniform grid of with a step $h>0$ 
\begin{equation}
	\eta_n = n h, \quad n\geq 0.
\end{equation}
If $n=0$, we immediately see from (\ref{eqn:EK}) that
\begin{equation}
	F_\alpha U(0) = \frac{U(0)}{\Gamma(2-\alpha)},
\end{equation}
therefore, we must focus on the case $n\geq 1$. To this end, change the integration variable in the definition (\ref{eqn:EK}) according to $z=s^{-B}\eta$ to obtain
\begin{equation}
	F_\alpha U(\eta_n) = \frac{\eta^{\frac{1-\alpha}{B}}}{B\Gamma(1-\alpha)} \int_{\eta_n}^\infty (\eta^{-\frac{1}{B}}-z^{-\frac{1}{B}})^{-\alpha} z^{-\frac{1}{B}-1} U(z) dz,
\end{equation}
where now the forward-nonlocal property of the EK operator is evident. That is to say, the value $F_\alpha U(\eta_n)$ depends on $U(z)$ for $z\geq \eta_n$. Since the integral is improper, in order to evaluate it numerically, we have to truncate it at some point, say $\eta_N$. We will choose the optimal value for this truncation later. Therefore,
\begin{equation}
	F_\alpha U(\eta_n) = \frac{\eta^{\frac{1-\alpha}{B}}}{B\Gamma(1-\alpha)} \int_{\eta_n}^{\eta_N} (\eta^{-\frac{1}{B}}-z^{-\frac{1}{B}})^{-\alpha} z^{-\frac{1}{B}-1} U(z) dz + R_N(\eta_n),
\end{equation}
with the remainder
\begin{equation}
	\label{eqn:Remainder}
	R_N(\eta_n) = \frac{\eta^{\frac{1-\alpha}{B}}}{B\Gamma(1-\alpha)} \int_{\eta_N}^\infty (\eta^{-\frac{1}{B}}-z^{-\frac{1}{B}})^{-\alpha} z^{-\frac{1}{B}-1} U(z) dz. 
\end{equation}
Now, we can write
\begin{equation}
	\label{eqn:FDisc}
	F_\alpha U(\eta_n) = \frac{\eta^{\frac{1-\alpha}{B}}}{B\Gamma(1-\alpha)} \sum_{i=n}^{N-1} \int_{\eta_i}^{\eta_{i+1}} (\eta^{-\frac{1}{B}}-z^{-\frac{1}{B}})^{-\alpha} z^{-\frac{1}{B}-1} U(z) dz + R_N(\eta_n),
\end{equation}
and approximate the function $U$ in the small interval $[\eta_i,\eta_{i+1})$. The two simplest choices are the rectangle and trapezoid approximation for which
\begin{equation}
	U(z) \approx U(\eta_{i+1}), \quad z\in [\eta_i,\eta_{i+1}),
\end{equation}
and 
\begin{equation}
	U(z) \approx U(\eta_{i}) + \frac{U(\eta_{i+1})-U(\eta_{i})}{h}(z-\eta_i), \quad z\in [\eta_i,\eta_{i+1}),
\end{equation}
respectively. Plugging the above into the EK integral reveals that 
\begin{equation}
	\label{eqn:EKDisc}
	F_\alpha U(\eta_n) \approx \widehat{F}_{\alpha,N} U(\eta_n) := \sum_{i=n}^{N} a^{(r,t)}_{in} U(\eta_i),
\end{equation}
with the following \emph{positive} weights that can be computed by a straightforward calculation
\begin{equation}
	\label{eqn:WeightsR}
	a^{(r)}_{in} = \begin{cases}
		\dfrac{1}{\Gamma(2-\alpha)}, & i = n = 0, \\
		0, & n = 0, \; 0<i\leq N, \\
		0, & i = n > 0, \\
		\dfrac{1}{\Gamma(2-\alpha)} \left[\left(1-\left(\dfrac{i}{n}\right)^{-\frac{1}{B}}\right)^{1-\alpha} - \left(1-\left(\dfrac{i-1}{n}\right)^{-\frac{1}{B}}\right)^{1-\alpha}\right], & n < i \leq N, \; n > 0.
	\end{cases}
\end{equation}
and
\begin{equation}
	\begin{split}
		\label{eqn:WeightsT}
		a^{(t)}_{in} &= \begin{cases}
			\dfrac{1}{\Gamma(2-\alpha)}, & i = n = 0, \\
			0, & n = 0, \; 0<i\leq N, \\
			a^{(r)}_{(n+1)n}-d_{nn}, & i = n > 0, \\
			d_{(i-1)n} - d_{in} + a^{(r)}_{(i+1)n}, & n < i \leq N-1, \; n > 0, \\
			d_{(M-1)n}, & i = N, \; n >0,
		\end{cases} \\
		d_{in} &= \frac{n}{\Gamma(1-\alpha)} \left[\beta\left(\left(\frac{i}{n}\right)^{-\frac{1}{B}}; 1-B, 1-\alpha\right)-\beta\left(\left(\frac{i+1}{n}\right)^{-\frac{1}{B}}; 1-B, 1-\alpha\right)\right] - i a^{(r)}_{(i+1)n},
	\end{split}
\end{equation}
where the superscripts $(r)$ and $(t)$ denote the rectangle and trapezoid rules, respectively. Here, $\beta(z;a,b)$ is the Euler incomplete beta function. The following result gives the error bounds for the discretization operator $\widehat{F}_{\alpha,N}$.
\begin{thm}
	Let $U:[0,\infty) \mapsto \mathbb{R}_+$ be a sufficiently bounded and smooth function. Moreover, set $N = \gamma n$, where 
	\begin{equation}
		\label{eqn:GammaOpt}
		\gamma = \begin{cases}
			[h^{-B}]+1, & \text{rectangle quadrature}, \\
			[h^{-2B}]+1, & \text{trapezoid quadrature}. 
		\end{cases}
	\end{equation}
	Then, we have the following error bounds
	\begin{equation}
		\|F_\alpha U(\eta_n) - \widehat{F}_{\alpha, N} U(\eta_n)\|_\infty \leq 
		\begin{cases}
			\left(\max_{z \geq \eta_N} |U(z)| + \max_{0\leq z \leq \eta_N} |U'(z)|\right)\dfrac{h}{\Gamma(2-\alpha)} , & \text{rectangle quadrature}, \\
			\left(\max_{z \geq \eta_N} |U(z)| + \frac{1}{2}\max_{0\leq z \leq \eta_N} |U''(z)|\right)\dfrac{h^2}{\Gamma(2-\alpha)}, & \text{trapezoid quadrature}. 
		\end{cases}
	\end{equation}
\end{thm}
\begin{proof}
	We will prove only the rectangle case, the other is completely analogous. By Taylor series, we immediately have
	\begin{equation}
		U(z) = U(\eta_{i+1}) + U'(\zeta_i) (\eta_{i+1}-z), \quad z\in [\eta_i,\eta_{i+1}),
	\end{equation}
	and plugging it into (\ref{eqn:FDisc}) yields
	\begin{equation}
		F_\alpha U(\eta_n) = \widehat{F}_{\alpha, N} U(\eta_n) + P_{N,h}(\eta_n) + R_N(\eta_n),
	\end{equation}
	where
	\begin{equation}
		|P_{N,h}(\eta_n)|\leq \max_{0\leq z \leq \eta_N} |U'(z)|\frac{\eta^{\frac{1-\alpha}{B}}}{B\Gamma(1-\alpha)} \sum_{i=n}^{N-1}  \int_{\eta_i}^{\eta_{i+1}} (\eta^{-\frac{1}{B}}-z^{-\frac{1}{B}})^{-\alpha} z^{-\frac{1}{B}-1} (\eta_{i+1}-z)dz.
	\end{equation}
	But since $|\eta_{i+1}-z| \leq h$ we can compute the sum explicitly and change back the integration variable $s^{-B}\eta_n = z$ to obtain the following
	\begin{equation}
		\begin{split}
			|P_{N,h}(\eta_n)|
			&\leq h\max_{0\leq z \leq \eta_N} |U'(z)|\frac{\eta^{\frac{1-\alpha}{B}}}{B\Gamma(1-\alpha)} \int_{\eta_n}^{\eta_{N}} (\eta^{-\frac{1}{B}}-z^{-\frac{1}{B}})^{-\alpha} z^{-\frac{1}{B}-1} (\eta_{i+1}-z)dz \\
			&=h\max_{0\leq z \leq \eta_N} |U'(z)| \frac{1}{\Gamma(1-\alpha)} \int_{(\frac{n}{N})^{1/B}}^1 (1-s)^{-\alpha} ds = \frac{h}{\Gamma(2-\alpha)} \max_{0\leq z \leq \eta_N} |U'(z)| \left(1-\gamma^{-\frac{1}{B}}\right)^{1-\alpha}.
		\end{split}
	\end{equation} 
	As for the truncation remainder (\ref{eqn:Remainder}) we can simply estimate
	\begin{equation}
		\begin{split}
			|R_N(\eta_n)| 
			&\leq \max_{z \geq \eta_N} |U(z)|\frac{\eta^{\frac{1-\alpha}{B}}}{B\Gamma(1-\alpha)} \int_{\eta_N}^\infty (\eta^{-\frac{1}{B}}-z^{-\frac{1}{B}})^{-\alpha} z^{-\frac{1}{B}-1} dz \\
			&= \max_{z \geq \eta_N} |U(z)|\frac{1}{\Gamma(1-\alpha)} \int_0^{\left(\frac{n}{N}\right)^{1/B}} (1-s)^{-\alpha} ds = \frac{\max_{z \geq \eta_N}|U(z)|}{\Gamma(2-\alpha)} \left(1-\left(1-\gamma^{-\frac{1}{B}}\right)^{1-\alpha}\right). 
		\end{split}
	\end{equation}
	Since
	\begin{equation}
		\left(1-\gamma^{-\frac{1}{B}}\right)^{1-\alpha} \leq 1, \quad \left(1-\left(1-\gamma^{-\frac{1}{B}}\right)^{1-\alpha}\right) \leq \gamma^{-\frac{1}{B}},
	\end{equation}
	where the second inequality follows from convexity, by our assumption (\ref{eqn:GammaOpt}) on $\gamma$ we have
	\begin{equation}
		|P_{N,h}(\eta_n)| + |R_N(\eta_n)| \leq \left(\max_{z \geq \eta_N} |U(z)| + \max_{0\leq z \leq \eta_N} |U'(z)|\right)\frac{h}{\Gamma(2-\alpha)}, 
	\end{equation}
	which concludes the proof. 
\end{proof}
As we can see, to obtain an optimal error, the truncation has to be chosen according to the grid spacing $h$. The optimality in this sense is associated with the same order of both remainders for $h\rightarrow 0^+$. Note also that the higher the order of the quadrature, the larger the interval over which we have to integrate. A numerical illustration of the above theorem can be presented by choosing a function with an explicitly known EK operator. Let $B=\alpha/2$ 
\begin{equation}
	U(\eta) = \min\left\{1, \eta^\mu\right\}, \quad \mu > 0,
\end{equation}
for which
\begin{equation}
	F_\alpha U(\eta) = \begin{cases}
		\frac{1-(1-\eta^{\frac{2}{\alpha}})^{1-\alpha}}{\Gamma(2-\alpha)} + \eta^{\mu} \left(\frac{\Gamma\left(1-\frac{\alpha \mu}{2}\right)}{\Gamma\left(2- \frac{\alpha (2+\mu)}{2}\right)}-\frac{\beta\left(\eta^\frac{2}{\alpha};1-\frac{\alpha \mu}{2},1-\alpha\right)}{\Gamma(1-\alpha)}\right), & 0\leq\eta < 1, \\
		\frac{1}{\Gamma(2-\alpha)}, & \eta \geq 1.
	\end{cases}
\end{equation}
We can now easily compute the discretization error. In Fig. \ref{fig:EKOrder} we depict the maximum error of approximating the EK operator with $\widehat{F}_{\alpha, N}$ with $N$ chosen according to the optimal choice (\ref{eqn:GammaOpt}). The error is plotted with respect to the grid spacing $h$, and the respective orders of approximation are clearly seen. As can be inferred, the graphs increase with a slope corresponding to the quadrature order. However, we note that due to the higher computational complexity of the trapezoid scheme in both function evaluations and the larger $\gamma$, this method is more expensive for the same $h$ compared to the simple rectangle quadrature. 

\begin{figure}
	\centering
	\includegraphics[scale = 0.8]{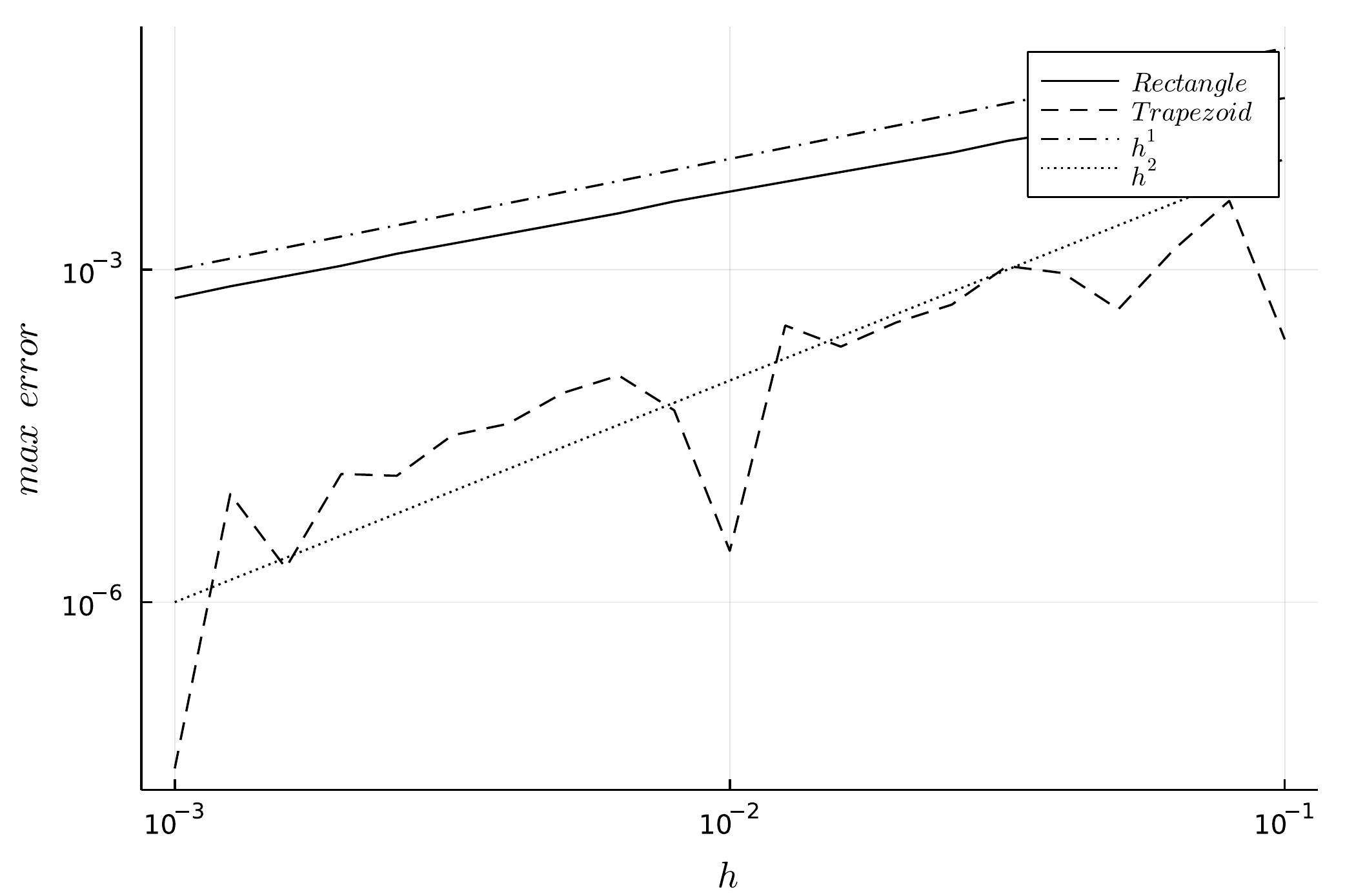}
	\caption{The maximal discretization error with respect to the grid spacing $h$ and $\alpha = 0.5$. Reference lines are added for comparison.}
	\label{fig:EKOrder}
\end{figure}

\subsection{Integro-differential equation}
We can now proceed to discretization of the main equation (\ref{eqn:PorousMediumD}). The strategy is to consider the integral form (\ref{eqn:PorousMediumInt}) rather than the original and solve the problem backwards. If some initial value is prescribed, say $U(0)=M$, we can use the shooting method and look for the zero of a function $\eta^* \mapsto M-U(0, \eta^*)$. 

For what follows, we fix the value of the wetting front $\eta^*$. Since we already know how to discretize the EK operator, it is just a matter of choosing the correct quadrature for the integral in (\ref{eqn:PorousMediumInt}). Similarly, as before, we choose either the rectangle or trapezoid methods. To this end, we naturally choose the integration horizon to $\eta^*$, that is, we choose $N$ in $\widehat{F}_{\alpha, N}$ according to
\begin{equation}
	\eta^* = \eta_N = N h,
\end{equation}
by possibly adjusting $h$ so that $N$ is an integer. Thus a discretization of the integral equation (\ref{eqn:PorousMediumInt}) can be found by splitting the integral into parts
\begin{equation}
	K(U(\eta_n)) = \sum_{j = n}^{N-1} \int_{\eta_j}^{\eta_{j+1}} G(\eta_n, z) F_\alpha U(z)dz.
\end{equation}
Now, approximating the EK operator on each subinterval by a constant or linear function yields the following discretization
\begin{equation}
	\label{eqn:Scheme0}
	K(U_n) = \sum_{j = n}^{N} b^{(r,t)}_{j n} \sum_{i = n}^N a^{(r,t)}_{ij} U_i = \sum_{i = n}^{N} \left(\sum_{j = n}^i  a^{(r,t)}_{i j} b^{(r,t)}_{jn} \right) U_j =: \sum_{i = n}^{N} c^{(r,t)}_{in} U_i,
\end{equation}
where by $U_i$ we have denoted the numerical approximation to the exact solution $U(\eta_i)$. Weights $b^{(r,t)}_{j n}$ correspond to the usual rectangle $(r)$ and trapezoid $(t)$ \textit{product} quadratures for the integral. Note that we do not approximate the kernel $G$ since it can be computed directly. More specifically,
\begin{equation}
	\label{eqn:WeightsRI}
	b^{(r)}_{nn} = 0, \quad b^{(r)}_{jn} = \frac{h^2}{2}\left((A+2B)(2j-1) - 2(A+B) n\right), \quad n<j\leq N,
\end{equation}
and
\begin{equation}
	\begin{split}
		\label{eqn:WeightsTI}
		b^{(t)}_{jn} &= \begin{cases}
			b^{(r)}_{(n+1)n}-\widehat{d}_{nn}, & j = n, \\
			\widehat{d}_{(j-1)n} - \widehat{d}_{jn} + b^{(r)}_{(j+1)n}, & n < i \leq N-1, \\
			\widehat{d}_{(M-1)n}, & j = N, \; \\
		\end{cases} \\
		\widehat{d}_{jn} &= \frac{h^2}{6}\left((A+2B)(3j-1) - 3(A+B) n\right),
	\end{split}
\end{equation}
for rectangle and trapezoid quadratures, respectively. The scheme's coefficients $c_{in}$ can then be computed as the convolution of $a_{ij}$ and $b_{jn}$, which can be done very efficiently with the use of Fast Fourier Transform. Note from (\ref{eqn:WeightsR}) and (\ref{eqn:WeightsRI}) that both rectangle weights vanish for $i = j = n$. This means that the numerical scheme (\ref{eqn:Scheme0}) is \emph{explicit} which makes the method simple and fast. That is, to compute $U_n$ we have to know the values of $U_i$ for $n<i\leq N$. Since $D$ is positive, its integral $K$ is increasing and, hence, has a well-defined inverse. Therefore, we can write
\begin{equation}
	U_n = K^{-1} \left(\sum_{i=n+1}^N c^{(r)}_{in} U_i\right).
\end{equation}
This is especially relevant for the important case of power-law diffusivity $K(u) = u^{m+1}/(m+1)$. Of course, when the analytical form of $K^{-1}$ is not available, it is probably better to apply a root-finding algorithm to (\ref{eqn:Scheme0}), especially in the implicit trapezoid case. The overall procedure is now $\eta$-stepping and computing the values of $U_n$ backwards for $n = N-1, N-2, N-3, \cdots, 1, 0$. 

Note that since the solution has a compact support terminating at $\eta=\eta^*= \eta_N$, we always have
\begin{equation}
	U_N = 0.
\end{equation}
The question arises how to start the scheme (\ref{eqn:Scheme0}) in order \emph{not to} obtain the trivial solution $U_i \equiv 0$. The answer is straightforward for the trapezoid method, for we have 
\begin{equation}
	\label{eqn:TerminalConditionT}
	K(U_{N-1}) = c^{(t)}_{(N-1)(N-1)} U_{N-1} = a^{(t)}_{(N-1)(N-1)}b^{(t)}_{(N-1)(N-1)} U_{N-1}, 
\end{equation}
which is a nonlinear algebraic equation to be solved for $U_{N-1}$. The explicit rectangle method does not have this form, and we have to prescribe the starting value in a different way. To this end, let us return to the integral equation (\ref{eqn:PorousMediumInt}) evaluated for $n = N-1$. If we use the constant function approximation at the \emph{left} endpoint of the interval $(\eta_{N-1},\eta_N)$ we obtain
\begin{equation}
	K(U(\eta_{N-1})) = \int_{\eta_{N-1}}^{\eta_N} G(\eta_{N-1}, z) F_\alpha U(z) dz \approx F_\alpha U(\eta_{N-1}) \int_{\eta_{N-1}}^{\eta_N} G(\eta_{N-1}, z)dz = F_\alpha U(\eta_{N-1}) b^{(r)}_{N(N-1)} 
\end{equation}
Doing the same left approximation in the rectangle quadrature for the EK operator (\ref{eqn:EKDisc}) yields the equation to be solved for the starting value $U_{N-1}$ 
\begin{equation}
	\label{eqn:TerminalConditionR}
	K(U_{N-1}) = a^{(r)}_{N(N-1)}b^{(r)}_{N(N-1)} U(\eta_{N-1}).
\end{equation}
Note the similarity with (\ref{eqn:TerminalConditionT}). For ease of implementation, the whole $\eta$-stepping scheme is summarized in Algorithm \ref{alg:MainScheme}. 

\begin{algorithm}
	\caption{Main $\eta$-stepping scheme for solving (\ref{eqn:PorousMediumD}) with (\ref{eqn:IC}) by the rectangle (r) or trapezoid (t) quadrature. The absolute tolerance is given by a fixed $0<\epsilon<1$.}
	\label{alg:MainScheme}
	\begin{algorithmic}
		\Require $N \in \mathbb{N}$
		\State Define weights $a^{(r,t)}$, $b^{(r,t)}$, and $c^{(r,t)}$ (use FFT) (\ref{eqn:WeightsR}), (\ref{eqn:WeightsT}), (\ref{eqn:Scheme0}), (\ref{eqn:WeightsRI}), (\ref{eqn:WeightsTI})
		\State Define the function $K=K(z)$ by (\ref{eqn:DiffusivityIntegrated})
		\State $U_{0:N} \gets 0$
		\While {$|U_0 - 1| \geq \epsilon$}
		\State Fix $\eta^*>0$ \Comment{A root finding solver does that.}
		\State $h \gets \eta^*/N$
		\State Solve $K(U_{N-1}) = a^{(r,t)}_{N(N-1)}b^{(r,t)}_{N(N-1)} U_{N-1}$ for $U_{N-1}$ 
		\For{$n = N-2:-1:0$}
		\State Solve $K(U_n) = \sum_{i=n}^{N-1} c^{(r,t)}_{in} U_i$ for $U_n$
		\EndFor
		\EndWhile
	\end{algorithmic}
\end{algorithm}

Having described the numerical scheme, we can proceed to proving that it is convergent. The main difficulty is the degeneracy of our equation, that is, the fact that $D(0) = 0$. In order to overcome it, we consider a family of regularizations $D_h$ of the diffusivity that converge to $D$ as $h\rightarrow 0$. For each $D_h$ we obtain a solution $U_{h,n}$ that converges to $U$ as we refine the grid. Let $D_h$ be a family of functions satisfying
\begin{equation}
	\label{eqn:DiffusivityReq}
	0< \epsilon(h) \leq D_h(z), \quad \epsilon(h) \rightarrow 0 \quad\text{as}\quad h\rightarrow 0,
\end{equation}
and 
\begin{equation}
	\|D_h - D\|_\infty \leq \epsilon(h).
\end{equation}
The following result states the convergence proof.
\begin{thm}
	Let $D_h$ be the family of regularizations of $D$ satisfying (\ref{eqn:DiffusivityReq}). Suppose that the weights of the quadrature (\ref{eqn:Scheme0}) satisfy $0<c_{in}\leq C h$ and that the and the quadrature of the integral in (\ref{eqn:DiffusivityIntegrated}) have order $p>0$. Then, there exists an $\epsilon(h)$ such that when $U_{h,n}$ is a solution of (\ref{eqn:Scheme0}) with $D_h$, and $U$ is an exact solution of (\ref{eqn:PorousMediumD}) we have
	\begin{equation}
		|U_{h,n} - U(\eta)| = O(h^{p-\delta} |\ln h|) \quad \text{when} \quad h\rightarrow 0, \quad nh\rightarrow \eta \in [0,\eta^*],
	\end{equation}
	where $0<\delta<1$ is arbitrary. 
\end{thm}
\begin{proof}
	Denote the error by $e_n := U(\eta_n) - U_{h,n}$ and let the quadrature error for the integral (\ref{eqn:PorousMediumInt}) be denoted by $\rho_n(h)$, that is,
	\begin{equation}
		\int_{\eta_n}^{\eta^*} G(\eta_n, z)F_\alpha U(z) dz = \sum_{i=n}^{N-1} c_{in} U_i + \rho_n(h) \quad \text{where} \quad |\rho_n(h)|\leq \rho(h) \rightarrow 0, \quad h\rightarrow 0. 
	\end{equation}
	Then, from this and (\ref{eqn:Scheme0}) we have
	\begin{equation}
		K_h(U(\eta_n)) - K_h(U_{h,n}) = \sum_{i=n}^{N-1} c_{in} e_i + \rho_n(h),
	\end{equation}
	where $K_h$ is the integral corresponding to $D_h$. Now, by the mean value theorem we can write $K_h(U(\eta_n)) - K_h(U_{n,h}) = K'(V_n) e_n = D_h(V_n) e_n$ for some intermediate value $V_n$, and hence
	\begin{equation}
		D_h(V_n) |e_n| \leq \sum_{i=n}^{N-1} c_{in} |e_i| + \rho(h).
	\end{equation}
	Now, by the construction of $D_h$ we can write
	\begin{equation}
		\epsilon(h)|e_n| \leq \sum_{i=n}^{N-1} c_{in} |e_i| + \rho(h),
	\end{equation}
	which is owing to the fact that $c_{in} \leq C h$
	\begin{equation}
		|e_n| \leq Ch\epsilon(h)^{-1}\sum_{i=n}^{N-1} |e_i| + \epsilon(h)^{-1}\rho(h),
	\end{equation}
	which is a form amenable for the discrete version of the Gr\"onwall inequality (for ex. heorem 7.1 from \cite{linz1985analytical} applied for $f_i := e_{N-i}$). Therefore, 
	\begin{equation}
		|e_n| \leq \epsilon(h)^{-1}\rho(h) \left(1+ Ch\epsilon(h)^{-1}\right)^n \leq \epsilon(h)^{-1}\rho(h) \left(1+ \frac{Cn h\epsilon(h)^{-1}}{n}\right)^n \leq \epsilon(h)^{-1}\rho(h) e^{C \eta^*\epsilon(h)^{-1} },
	\end{equation}
	since $n h \rightarrow \eta \leq \eta^*$. Now, if we choose 
	\begin{equation}
		\epsilon(h) = C \eta^* \left(\delta\ln \frac{1}{h}\right)^{-1} \rightarrow 0 \quad \text{as} \quad h\rightarrow 0,  
	\end{equation}
	for some arbitrary $0<\delta<1$, we obtain
	\begin{equation}
		|e_n| \leq \frac{\delta}{C \eta^*}\rho(h) h^{-\delta} |\ln h|.
	\end{equation}
	If now $\rho(h) = O(h^p)$ with $p>0$ and $h\rightarrow 0$, we have
	\begin{equation}
		|e_n| = O(h^{p-\delta} |\ln h|), \quad h\rightarrow 0,
	\end{equation}
	what finishes the proof. 
\end{proof}
From the above proof we thus see that the order of the scheme for the regularized solution is almost $p$, that is, the order of the quadrature for (\ref{eqn:PorousMediumInt}). The actual order is less by an arbitrary small number $\delta$ and a logarithmic factor. 

We illustrate our theory by some numerical experiments. In what follows, we always choose the rectangle scheme in approximating the solution. Our simulations indicated that although the trapezoidal method is superior when discretizing the pure EK operator (\ref{eqn:EK}) it is very expensive when applied to the nonlinear equation (\ref{eqn:PorousMediumD}). This computational cost comes from a large number of special functions needed to calculate the weights of the trapezoid method (\ref{eqn:WeightsT}) - especially the incomplete beta function. In effect, the temporal and spatial complexity of the algorithm can be prohibitively large. Moreover, the method is implicit without a significant stability gain, and thus requires solving a nonlinear equation in each iteration step. As a benchmark, we have calculated the time ratio of computations needed to obtain the wetting front position with the diffusivity $D(u) = u^2$ for number of subdivisions $N=2^8$ for different values of $\alpha$. In Tab. \ref{tab:TimeRatio} we present the quantity 
\begin{equation}\label{eqn:TimeRatio}
	\tau = \frac{\text{time of computations for trapezoid method}}{\text{time of computations for the rectangle method}}.
\end{equation}
Immediately we see that computations with the trapezoidal method are at least \emph{one hundred} times slower than with the rectangle method. We can conclude that the increase in accuracy for the second order method does not compensate the high increase in computational cost. We have thus decided that a less accurate but much faster explicit rectangle method will be the scheme of choice. An efficient second-order explicit scheme for the power-law case, i.e. $D(u) \propto u^m$ has been devised in \cite{okrasinska2022second} by different means that cannot be generalized to the arbitrary diffusivity. 

\begin{table}
	\centering
	\begin{tabular}{cccccc}
		\toprule
		$\alpha$ & 0.1 & 0.25 & 0.5 & 0.75 & 0.9 \\
		\midrule
		$\tau$ & 91 & 230 & 168 & 197 & 380 \\
		\bottomrule
	\end{tabular}
	\caption{Time ratio $\tau$ defined in (\ref{eqn:TimeRatio}) for computing the wetting front position with $D(u) = u^2$, $N=2^8$ with the trapezoid and rectangle method. }
	\label{tab:TimeRatio}
\end{table}

\begin{figure}
	\centering
	\includegraphics[scale = 0.8]{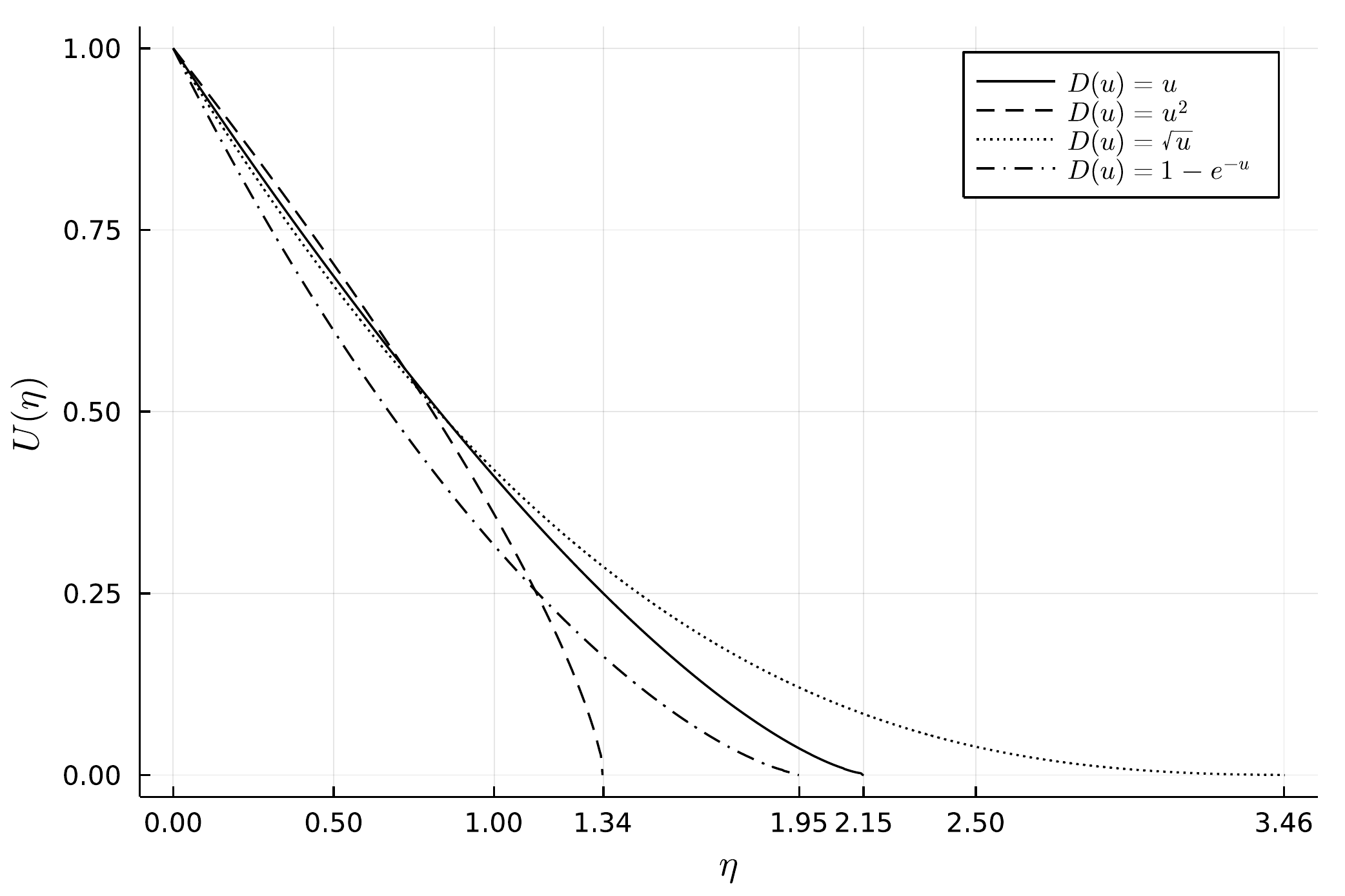}
	\caption{An exemplary plot of solutions to (\ref{eqn:PorousMediumD}) with $U(0) = 0$ for $\alpha = 0.5$ calculate with the Algorithm \ref{alg:MainScheme}. Different diffusivities are indicated in the legend. }
	\label{fig:Solutions}
\end{figure}

In what follows we always solve (\ref{eqn:PorousMediumD}) with the initial condition $U(0) = 1$ with the Algorithm \ref{alg:MainScheme} where we use Newton's iteration for finding $\eta^*$. Some exemplary plots of typical solutions of the porous medium equation are depicted in Fig. \ref{fig:Solutions} for several choices of the diffusivity
\begin{equation}
	\label{eqn:Diffusivities}
	D_{BC}(u) = u^m, \quad D_{exp}(u) = 1-e^{-u}, \quad m \geq 1,
\end{equation}
where the first choice is the typical porous medium power-type diffusivity (in hydrology known as the Brooks-Correy model) and the second is the exponential formula. By a simple limit test, we can verify that the necessary condition for the compact support (\ref{eqn:DiffusivityInt}) is satisfied for each of the diffusivities. 

To illustrate convergence, we present two tests. The first one is an empirical calculation of the convergence order using extrapolation (Aitken's method)
\begin{equation}
	\label{eqn:Aitken}
	\text{order } \approx \log_2 \frac{|\eta^*_{(2N)}-\eta^*_{(N)}|}{|\eta^*_{(4N)}-\eta^*_{(2N)}|},
\end{equation}
in which we compare the wetting front (the worst-case scenario) calculated for different numbers of subdivisions $N$ (hence, twice or quadruple smaller grid spacing $h$). The results for two different diffusivities are presented in the Tab. \ref{tab:Order}. As can be seen, the numerical estimates are consistent with the theoretical predictions that the rectangle quadrature should retain its order. This is not obvious since it is widely known that even for the \emph{linear diffusion}, the discretizations of fractional derivatives may converge with lower order than $1$ depending on the regularity of the solution (for a comprehensive account, see \cite{stynes2017error}. We can see that for $\alpha = 0.75$ the method converged even faster. 

\begin{table}
	\centering
	\begin{tabular}{cccccc}
		\toprule
		$\alpha$ & $0.1$ & $0.25$ & $0.5$ & $0.75$ & $0.9$ \\
		\midrule
		order for $D_{BC}$ & 0.95 & 0.98 & 0.99 & 1.255 & 0.97 \\
		order for $D_{exp}$ & 0.96 & 0.98 & 0.99 & 1.08 & 0.98 \\
		\bottomrule
	\end{tabular}
	\caption{Estimated order of quadrature for the scheme (\ref{eqn:Scheme0}) based on extrapolation (\ref{eqn:Aitken}) applied for calculating the wetting front $\eta^*$. Two diffusivities has been chosen from (\ref{eqn:Diffusivities}): $D_{BC}$ with $m=1$ and $D_{exp}$. The base number of iteration is $N = 300$.}
	\label{tab:Order}
\end{table}

The second test we conduct is once again based on finding the wetting front, but not in the classical case, since then the exact values are available (see \cite{okrasinski1993power}). The results are gathered in the Tab. \ref{tab:Wetting}. The results are decent based on the fact that a small number of steps have been taken. The error decays at a rate $O(N^{-1})$ as $N\rightarrow \infty$. This verifies the fact that our method is convergent even in the classical case. However, if one wants to conduct precision calculations on the wetting front for power type diffusivity, it is recommended to use our second order method \cite{okrasinska2022second}. The present one has the advantage of being fast and robust when it comes to choosing a general form of diffusivity. 

\begin{table}
	\centering
	\begin{tabular}{ccccccc}
		\toprule
		N & 10 & 50 & 100 & 200 & 500 & 1000 \\
		\midrule
		error & $5.5\times 10^{-2}$ & $3.3 \times 10^{-2}$ & $2.2\times 10^{-2}$ & $1.3\times 10^{-2}$ & $7.0\times 10^{-3}$ & $4.0\times 10^{-3}$ \\
		\bottomrule
	\end{tabular}
	\caption{Absolute errors of calculating the wetting front $\eta^*$ for different values of subdivisions of the interval $[0,\eta^*]$ for $D_BC(u) = u$ and $\alpha = 1$. The reference exact values were taken from \cite{okrasinski1993power}.}
	\label{tab:Wetting}
\end{table}

\section{Conclusion}
The time-fractional porous medium equation models several important experimental settings in material science, hydrology, and construction engineering. We have proved that the problem with general diffusivity has a unique solution that has a self-similar form. The main role was played here by the Erd\'elyi-Kober fractional operator and its careful analysis. On the practical side, we have devised a robust numerical method that can be easily used by practitioners. 

In our future work, we plan to resign from the small initial value requirement and to consider a generalized version of (\ref{eqn:MainPDE}) where we will allow for a nonlocal in space operator. This will enlarge the number of possible modeling situations and include the superdiffusive case, which has also been found in many experiments. 

\section*{Acknowledgement}
Ł.P. has been supported by the National Science Centre, Poland (NCN) under the grant Sonata Bis with a number NCN 2020/38/E/ST1/00153.

\bibliography{biblio}
\bibliographystyle{plain}

\end{document}